\documentclass[10pt,reqno]{amsart}

\usepackage{a4wide}
\usepackage[utf8]{inputenc}
\usepackage{t1enc,mathrsfs}
\usepackage{dsfont}   
\usepackage{todonotes}
\usepackage[all]{xy}  
\usepackage{upgreek}

\newtheorem{lemma}{Lemma}[section]
\newtheorem{proposition}[lemma]{Proposition}
\newtheorem{corollary}[lemma]{Corollary}
\newtheorem{theorem}[lemma]{Theorem}
\newtheorem*{claim}{Claim}
\theoremstyle{definition}
\newtheorem{definition}[lemma]{Definition}
\newtheorem{example}[lemma]{Example}
\newtheorem{remark}[lemma]{Remark}

\usepackage{cleveref} 

\creflabelformat{enumi}{#2(#1)#3}

\crefname{section}{Section}{Sections}
\crefformat{section}{#2Section~#1#3}
\Crefformat{section}{#2Section~#1#3}

\crefname{subsection}{\S}{\S\S}
\AtBeginDocument{%
  \crefformat{subsection}{#2\S#1#3}%
  \Crefformat{subsection}{#2\S#1#3}%
}
\crefname{definition}{definition}{definitions}
\crefformat{definition}{#2definition~#1#3}
\Crefformat{definition}{#2Definition~#1#3}

\crefname{ex}{example}{examples}
\crefformat{example}{#2example~#1#3}
\Crefformat{example}{#2Example~#1#3}

\crefname{remark}{remark}{remarks}
\crefformat{remark}{#2remark~#1#3}
\Crefformat{remark}{#2Remark~#1#3}

\crefname{lemma}{lemma}{lemmas}
\crefformat{lemma}{#2lemma~#1#3}
\Crefformat{lemma}{#2Lemma~#1#3}

\crefname{proposition}{proposition}{propositions}
\crefformat{proposition}{#2proposition~#1#3}
\Crefformat{proposition}{#2Proposition~#1#3}

\crefname{theorem}{theorem}{theorems}
\crefformat{theorem}{#2theorem~#1#3}
\Crefformat{theorem}{#2Theorem~#1#3}

\crefname{enumi}{}{}
\crefformat{enumi}{(#2#1#3)}
\Crefformat{enumi}{(#2#1#3)}

\crefname{equation}{}{}
\crefformat{equation}{(#2#1#3)}
\Crefformat{equation}{\rm(#2#1#3)}

\numberwithin{equation}{section}

\newcommand{\cst}{\ensuremath{\mathrm{C}^*}}
\newcommand{\cat}[1]{\textsc{#1}}
\newcommand{\I}{\mathds{1}}

\newcommand{\cA}{\mathcal{A}}
\newcommand{\cH}{\mathcal{H}}
\newcommand{\cI}{\mathcal{I}}
\newcommand{\cJ}{\mathcal{J}}
\newcommand{\cC}{\mathcal{C}}
\newcommand{\cT}{\mathcal{T}}

\newcommand{\bC}{\mathbb{C}}
\newcommand{\bR}{\mathbb{R}}
\newcommand{\bT}{\mathbb{T}}
\newcommand{\bG}{\mathbb{G}}
\newcommand{\bH}{\mathbb{H}}
\newcommand{\bS}{\mathbb{S}}
\newcommand{\bZ}{\mathbb{Z}}

\DeclareMathOperator{\cK}{\mathcal{K}}
\DeclareMathOperator{\cB}{\mathcal{B}}
\DeclareMathOperator{\bK}{\mathbb{K}}
\DeclareMathOperator{\CCR}{CCR}
\DeclareMathOperator{\C}{C}
\DeclareMathOperator{\Cu}{C_{\text{\tiny{\rm u\!}}}}
\DeclareMathOperator{\Cr}{C_{\text{\tiny{\rm r\!}}}}
\DeclareMathOperator{\li}{\mathnormal{L}^\infty}
\DeclareMathOperator{\Pol}{Pol}
\DeclareMathOperator{\Irr}{Irr}

\title{Compact quantum group structures on type-I $\mathrm{C}^*$-algebras}

\author{Alexandru Chirvasitu}
\address{Department of Mathematics, University at Buffalo} \email{achirvas@buffalo.edu}

\author{Jacek Krajczok}
\address{Institute of Mathematics of the Polish Academy of Sciences, Warsaw}\email{jkrajczok@impan.pl}

\author{Piotr M.~So{\l}tan}
\address{Department of Mathematical Methods in Physics, Faculty of Physics, University of Warsaw}\email{piotr.soltan@fuw.edu.pl}

\begin{document}

\begin{abstract}
  We prove a number of results having to do with equipping type-I $\mathrm{C}^*$-algebras with compact quantum group structures, the two main ones being that such a compact quantum group is necessarily co-amenable, and that if the $\mathrm{C}^*$-algebra in question is an extension of a non-zero finite direct sum of elementary $\mathrm{C}^*$-algebras by a commutative unital $\mathrm{C}^*$-algebra then it must be finite-dimensional. 
\end{abstract}



\maketitle

\section*{Introduction}

The theory of locally compact quantum groups is inextricably connected to the theory of operator algebras. In fact, paraphrasing S.L.~Woronowicz \cite[Section 0]{psudosp}, any theorem on locally compact quantum groups is one on \cst-algebras. In the present paper we will focus on some of the interplay between the theory of compact quantum groups and operator algebras. Examples of such an interplay are motivated by results such as the well known equivalence between amenability of a discrete group $\Gamma$ and nuclearity of the \cst-algebra $\cst_\text{\tiny{r}\!}(\Gamma)$ (\cite[Theorem 4.2]{lance}). This particular fact has been generalized to compact quantum groups (i.e.~duals of discrete quantum groups, see \cite[Section 3]{pw-qlor}) of Kac type by Tomatsu in \cite[Corollary 1.2]{tomatsuAmen} and in a weakened form to all compact quantum groups in \cite[Theorem 3.9]{tomatsuAmen} (see also \cite[Theorem 3.3]{BeTu-lcqgAmen}). Some of these topics were pursued further e.g.~in \cite{boca,dlrz} in the language of quantum group actions as well as in \cite{boca,SoVi,CraNeu} in the locally compact case.

The moral of the above-mentioned research activity is that one can learn about certain ``group-theoretical'' properties of a compact quantum group $\bG$ by studying purely operator theoretic properties of the \cst-algebra $\C(\bG)$. Furthermore one can often show that certain \cst-algebras do not admit a compact quantum group structure solely on the basis of some of their properties as \cst-algebras. Examples of such results are given in \cite{sol-wo,sol-when} and also \cite{sol-qsemi} and most recently \cite{KrajSoDisk}. In this last paper the second and third author show that the \cst-algebra known as the \emph{Toeplitz algebra} (the \cst-algebra generated by an isometry) does not admit a structure of a compact quantum group. The main tools are built out of certain direct integral decompositions available for so called \emph{type-I} quantum groups, i.e.~locally compact quantum groups whose universal quantum group \cst-algebra is of type I with particular emphasis on \cst-algebras of type I with \emph{discrete CCR ideal} (see \Cref{se:prel}).

In the present paper the techniques of \cite{KrajSoDisk} are vastly generalized and applied to a number of problems. Moreover the direct integral decompositions of representations (and other objects) are avoided. In the preliminary \Cref{se:prel} we introduce our basic tools, recall certain objects such as the CCR ideal of a \cst-algebra and prove a number of lemmas concerning implementation of automorphisms on \cst-algebras with discrete CCR ideals. The main result of \Cref{se:adm} is \Cref{th:autocoam} which says that a compact quantum group $\bG$ with $\C(\bG)$ of type I must be co-amenable (\cite[Section 1]{bmt-coam}). Along the way we prove a number of results about the scaling group of a compact quantum group which allow to reprove the result of Daws (\cite{daws-bohr}) about automatic admissibility of finite dimensional representations of any discrete quantum group (cf.~\cite[Section 2.2]{sol-bohr}).

In the final \Cref{se:cpctext} we discuss compact quantum group structures on \cst-algebras which are extensions of a finite direct sum of algebras of compact operators by a commutative \cst-algebra. Examples of such \cst-algebras occur quite frequently in non-commutative geometry and include the Podle\'s spheres (\cite{pod-sph,dab-gar}), the quantum real projective plane (\cite[Section 3.2]{JeoSzy-proj} and some weighted quantum projective space (\cite[Section 3]{BrzSzy}) and many others (see \Cref{se:cpctext}). We show that such \cst-algebras do not admit any compact quantum group structure which answers several questions left open in \cite{sol-wo,sol-when} and provides a number of fresh examples of naturally occurring quantum spaces with this property.

Our exposition is based on a number of standard references. Thus we refer to classic texts such as \cite{dix-cast,arv-inv} for all necessary background on \cst-algebras and to \cite{pseudogr,nt-bk} for the theory of compact quantum groups. We have tried to keep the terminology and notation consistent with recent trends and as self-explanatory as possible. In particular, for a compact quantum group $\bG$ we denote by $\C(\bG)$ the (usually non-commutative) \cst-algebra playing the role of the algebra of continuous functions on $\bG$. The symbol $\Irr(\bG)$ will denote the set of equivalence classes of irreducible representations of $\bG$ and for a class $\alpha\in\Irr(\bG)$ the dimension of $\alpha$ will be denoted by $n_\alpha$. Since in the theory of compact quantum groups we allow the \cst-algebras $\C(\bG)$ to be sitting strictly between the reduced and universal versions (see \cite{bmt-coam}), we will write $\Cr(\bG)$ and $\Cu(\bG)$ for these two distinguished completions of the canonical Hopf $*$-algebra $\Pol(\bG)$ inside $\C(\bG)$.

\subsection*{Acknowledgments}

AC is grateful for funding through NSF grants DMS-1801011 and DMS-2001128.

The second and third authors were partially supported by the Polish National Agency for the Academic Exchange, Polonium grant PPN/BIL/2018/1/00197 as well as by the FWO–PAS project VS02619N: von Neumann algebras arising from quantum symmetries.

\section{Preliminaries}\label{se:prel}

All \cst-algebras are unital except when we specify otherwise, or with obvious exceptions such as the algebra of compact operators on an infinite-dimensional Hilbert space.

We denote by $\widehat{\cA}$ the \emph{spectrum} of the \cst-algebra $\cA$, i.e.~the set of equivalence classes of irreducible non-zero representations (\cite[\S 2.2.1 \& \S 2.3.2]{dix-cast}).

For a Hilbert space $H$ we denote by $\cK(H)$ the algebra of compact operators on $H$. Furthermore for a family $\cH=\{H_{\lambda}\}_{\lambda\in \Lambda}$ of Hilbert spaces we set
\begin{equation}\label{eq:ckh}
  \bK(\cH):=
  \operatorname{c_0}\text{-}\bigoplus_{\lambda\in\Lambda}\cK(H_{\lambda})
\end{equation}
the algebra of compact operators in
\begin{equation*}
  H=\bigoplus_{\lambda\in \Lambda} H_{\lambda}
\end{equation*}
preserving that direct sum decomposition. In general, for non-unital \cst-algebras $\cA$, we write $\cA^+$ for the minimal unitization of $\cA$.

\subsection{Type-I \cst-algebras}\label{subse:type1}

We will work with \emph{type-I} \cst-algebras in the sense of \cite{glm}, which provides numerous equivalent characterizations. Textbook sources are \cite[Chapter 9]{dix-cast} and \cite[\S 1.5 and Chapters 2 and 4]{arv-inv}.

Recall e.g.~from \cite[discussion preceding Definition 1.5.3]{arv-inv} the following definition:

\begin{definition}\label[definition]{def:ccr} 
  For a \cst-algebra $\cA$, the \emph{CCR ideal} $\CCR(\cA)$ is the intersection, over all irreducible representations
  \begin{equation*}
    \pi:\cA\longrightarrow \cB(H),
  \end{equation*}
  of the pre-images $\pi^{-1}(\cK(H))$ of the ideal $\cK(H)$ of compact operators on $H$.
\end{definition}

In other words, $\CCR(\cA)$ consists of those elements that are compact in every irreducible representation.

\begin{definition}\label[definition]{def:discccr} 
  A (typically type-I) \cst-algebra $\cA$ is said to have \emph{discrete CCR ideal} if its CCR ideal $\CCR(\cA)$ is of the form $\bK(\cH)$ as in \Cref{eq:ckh} for some family $\cH=\{H_{\lambda}\}_{\lambda\in\Lambda}$ of Hilbert spaces.

  We occasionally also say $\cA$ is \emph{discrete-CCR} or \emph{CCR-discrete} for brevity, though note that this does \emph{not} mean it is CCR!
\end{definition}

Now let $\cA$ be a type-I discrete-CCR \cst-algebra, with
\begin{equation*}
  \CCR(\cA) = \bK(\cH),\quad \cH=\{H_{\lambda}\}_{\lambda\in\Lambda}
\end{equation*}
and set
\begin{equation}\label{eq:hdec}
  H:=\bigoplus_{\lambda\in\Lambda}H_{\lambda}.
\end{equation}

The ideal $\bK(\cH)\subset \cA$ is represented in the obvious fashion on $H$ with each component $\cK(H_{\lambda})$ acting naturally on $H_{\lambda}$. This representation $\rho_0:\bK(\cH)\to\cB(H)$ extends uniquely to a representation
\begin{equation*}
  \rho:\cA\longrightarrow \cB(H)
\end{equation*}
by e.g.~\cite[Theorem 1.3.4]{arv-inv}.

\begin{lemma}\label{le:autoinnergen}
  Let $\cA$ be a type-I \cst-algebra with $\CCR(\cA)$ of the form \Cref{eq:ckh}. Representation $\rho$ is faithful and every automorphism of $\cA$ is given by conjugation by some unitary $U\in\operatorname{U}(H)$.

  Furthermore, if the family $\cH$ is a singleton then that unitary is unique up to scaling by $\bT^1$.
\end{lemma}
\begin{proof}
  Recall from \cite[pp.~14--15]{arv-inv} that the representation $\rho$ is constructed from $\rho_0:\operatorname{CCR}(\mathcal{A})=\mathbb{K}(\mathcal{H})\to\cB(H)$ as follows: an element $a\in\mathcal{A}$ is mapped to the unique element $\rho(a)$ of $\cB(H)$ such that $\rho(a)\rho_0(x)=\rho_0(ax)$ for all $x\in\operatorname{CCR}(\mathcal{A})$. By construction $\rho=\bigoplus\limits_{\scriptscriptstyle\lambda\in\Lambda}\rho^\lambda$, where $\rho^\lambda$ is constructed analogously from $\rho_0^\lambda:\mathcal{K}(H_{\lambda})\to\cB(H_{\lambda})$. Since each $\rho^\lambda_0$ is irreducible, so is each $\rho^\lambda$ (\cite[Theorem 1.3.4]{arv-inv}).

  We now note that the CCR-ideal $\CCR(\mathcal{A})$ is essential. Indeed, $\CCR(\cA)$ is the largest CCR ideal in $\cA$ (cf.~\cite[p.~24]{arv-inv}). But every ideal in $\cA$ is type I and every type-I \cst-algebra contains a non-zero CCR ideal, so any non-zero ideal of $\cA$ must have a non-zero intersection with $\CCR(\cA)$. It follows that $\rho$ is faithful.

  Now for any $\alpha\in\operatorname{Aut}(\mathcal{A})$ the representation $\rho_0^\lambda\circ\alpha$ is equivalent to $\rho_0^\lambda$, so by \cite[Thm. 1.3.4]{arv-inv} $\rho^\lambda$ is equivalent to $\rho^\lambda\circ\alpha$ (because for $a\in\mathcal{A}$ and $x\in\CCR(\mathcal{A})$ we have $(\rho^\lambda\circ\alpha)(a)(\rho^\lambda_0\circ\alpha)(x)=(\rho^\lambda_0\circ\alpha)(ax)$). For each $\lambda$ let $U_{\lambda}$ be a unitary implementing the equivalence. Then $U=\bigoplus\limits_{\scriptscriptstyle\lambda\in\Lambda}U_{\lambda}$ implements equivalence between $\rho\circ\alpha$ and $\rho$.

  As for uniqueness, it follows from the fact that when $\cH=\{H_{\lambda}\}_{\lambda\in\Lambda}$ is a singleton the representation $\rho$ is irreducible, and hence the only self-intertwiners of $\rho$ are the scalars.
\end{proof}

Of more interest to us, however, will be one-parameter automorphism groups (where we can also recover some measure of uniqueness):

\begin{lemma}\label{le:1paraminnergen}
  Let $\cA$ be a type-I \cst-algebra with $\CCR(\cA)$ of the form \Cref{eq:ckh}. A one-parameter automorphism group $(\alpha_s)_{s\in\bR}$ of $\cA$ is given by conjugation by a one-parameter unitary group
  \begin{equation*}
    \bR\ni s\longmapsto U_s\in \prod_{\lambda\in\Lambda}\operatorname{U}(H_{\lambda})\subset\operatorname{U}(H),
  \end{equation*}
  preserving the decomposition \Cref{eq:hdec}, unique up to scaling by an individual character $\chi_{\lambda}:\bR\to \bT$ on each $H_{\lambda}$.
\end{lemma}
\begin{proof}
  Every automorphism of $\cA$ will permute the summands $\cK(H_{\lambda})$ of $\bK(\cH)$, so a one-parameter group will preserve each summand by continuity. But this means that on each $H_{\lambda}$ the automorphisms $(\alpha_s)_{s\in\bR}$ are given by conjugation by a \emph{projective} unitary representation (\cite[Chapter VII, Section 2]{var}) of $\bR$ on $H_{\lambda}$. Since projective representations of $\bR$ lift to plain unitary representations, for each $s$ we have
  \begin{equation*}
    \bigl.\alpha_s\bigr|_{\cB(H_{\lambda})} = \text{conjugation by }b^{is}
  \end{equation*}
  for a possibly-unbounded positive self-adjoint non-singular operator $b$ on $H_{\lambda}$. This lift is moreover unique up to multiplication by a character $\bR\to\operatorname{U}(H_{\lambda})$ because $\cK(H_{\lambda})$ acts irreducibly on $H_{\lambda}$.
\end{proof}

\subsection{Scaling groups}\label{subse:scale}

We recall the following well-known observation.

\begin{lemma}\label{le:scalepres}
  Let $\bG$ and $\bH$ be compact quantum groups. Any Hopf $*$-homomorphism
  \begin{equation*}
    \phi:\C(\bG)\longrightarrow \C(\bH)
  \end{equation*}
(i.e.~a unital $*$-homomorphism satisfying $\Delta_{\bH}\circ\phi=(\phi\otimes \phi)\circ \Delta_{\bG}$) intertwines scaling groups, in the sense that
  \begin{equation*}
  \phi\circ\tau^{\bG}_{s}(a)=\tau_{s}^{\bH}\circ \phi(a),
  \quad\forall \:s\in\bR,\,a\in\Pol(\bG).
  \end{equation*}

\end{lemma}

\begin{proof}
Assume first that $\C(\bG)=\Cu(\bG),\C(\bH)=\Cu(\bH)$ are universal versions of the algebras of continuous functions. In this case our lemma is simply a reformulation of \cite[Proposition 3.10, equation (20)]{mrw}.\\
Consider now the general case. Observe that $\phi$ restricts to a map $\Pol(\bG)\rightarrow \Pol(\bH)$, hence by the universal property of $\Cu(\bG)$ we can extend $\phi|_{\Pol(\bG)}$ to a $*$-homomorphism $\tilde{\phi}\colon\Cu(\bG)\rightarrow\Cu(\bH)$.  Clearly $\tilde{\phi}$ is a Hopf $*$-homomorphism, hence by the above argument $\tilde{\phi}$ intertwines scaling groups. As $\phi$ and $\tilde{\phi}$ are equal on $\Pol(\bG)$ and the canonical morphisms $\Cu(\bG)\rightarrow\C(\bG),\:\Cu(\bH)\rightarrow\C(\bH)$ intertwine scaling groups, we arrive at the claim.
\end{proof}

\section{Admissibility and co-amenability}\label{se:adm}

Throughout the discussion we denote by $\bG$ a compact quantum group and by $\Gamma=\widehat{\bG}$ its discrete quantum dual. The following observation will be put to use repeatedly; it is \cite[lemma 2.3]{cs}, and it follows from \Cref{le:scalepres} upon noting that finite-dimensional representations factor through Kac quotients. 

\begin{proposition}\label[proposition]{pr:findiminv} 
  Every finite-dimensional representation $\rho:\cA\to M_n$ of the CQG algebra $\cA=\C(\bG)$ is invariant under the scaling group $(\tau_s)_{s\in\bR}$ of $\bG$, in the sense that
  \begin{equation*}
    \rho \circ\tau_s(a) = \rho(a),\quad \forall\:s\in \bR,\,a\in\Pol(\bG).
  \end{equation*}
  \qedhere
\end{proposition}


\Cref{pr:findiminv} has a number of consequences. First, note the following generalization.

\begin{corollary}
  Let $\cB$ be a \cst-algebra all of whose irreducible representations are finite-di\-men\-sion\-al and $\cA=\C(\bG)$ for a compact quantum group $\bG$. Then, every morphism $\rho:\cA\to \cB$ is invariant under the scaling group $(\tau_s)_{s\in\bR}$ of $\bG$.
\end{corollary}
\begin{proof}
  Indeed, it follows from \Cref{pr:findiminv} that for every irreducible representation $\pi:\cB\to M_n$ the composition $\pi\circ \rho$ is invariant under $\tau$. The conclusion follows from the fact that the direct sum of all $\pi$ is faithful on $\cB$ (i.e.~every \cst-algebra embeds into the direct sum of its irreducible representations).
\end{proof}

Secondly, we obtain the following alternative proof of \cite[Corollary 6.6]{daws-bohr} or \cite[Proposition 3.3]{vis-mix} which concerns admissibility of finite representations of discrete quantum. The relevant terminology is explained in \cite{sol-bohr,daws-bohr,dds}.

\begin{theorem}
  Every finite-dimensional unitary representation of a discrete quantum group is admissible.
\end{theorem}
\begin{proof}
  Let $\bG$ be a compact quantum group and denote by $\Gamma$ the dual of $\bG$. Furthermore put $\cA:=\Cu(\bG)$. As explained in \cite[Theorem 3.4]{pw-qlor} (cf.~\cite[Proposition 5.3]{kust-univ}, \cite[Section 5]{fmutqg2}), a unitary representation of $\Gamma$ on $\bC^n$ is defined by a morphism $\rho:\cA\to M_n$.

  \Cref{pr:findiminv} ensures that $\rho$ is invariant under the scaling group of $\bG$. But then, by \cite[Proposition 3.2 and Remark 3.4]{dds}, the representation of $\Gamma$ associated to $\rho$ will be admissible.
\end{proof}

Next, we have the following sufficient criterion for the co-amenability of a compact quantum group $\bG$. It appears as \cite[Proposition 2.5]{cs}, and we include a slightly different proof here.

\begin{theorem}\label[theorem]{th:coamcrit} 
  A compact quantum group $\bG$ is co-amenable if and only if the reduced algebra $\Cr(\bG)$ admits a morphism $\rho:\Cr(\bG)\to M_n$ to a finite-dimensional \cst-algebra.
\end{theorem}
\begin{proof}
  Co-amenability means the counit is bounded on $\Cr(\bG)$, so only the backwards implication `$\Leftarrow$' is interesting. We know from \Cref{pr:findiminv} that $\rho$ is invariant under the scaling group $\tau_s$, $s\in \bR$, so by analytic continuation its restriction to the dense Hopf $*$-subalgebra $\Pol(\bG)\subset\Cr(\bG)$ is invariant under the squared antipode
  \begin{equation*}
    S^2 = \tau_{-i}
  \end{equation*}
  (\cite[p.32]{nt-bk}). Once we have $S^2$-invariance, co-amenability follows from \cite[Theorem 4.4]{bmt1}.
\end{proof}

\begin{remark}
  \Cref{th:coamcrit} generalizes \cite[Theorem 2.8]{bmt-coam}, which requires the existence of a bounded \emph{character}, and strengthens \cite[Theorem 4.4]{bmt1} by removing the $S^2$-invariance hypothesis (which is automatic).
\end{remark}

For future reference, we also record the following description of the Kac quotient of a CQG algebra.

\begin{proposition}\label[proposition]{pr:kac}
Let $\cA=\C(\bG)$ and $(\tau_s)_{s\in\bR}$ the corresponding scaling group.

The Kac quotient $\cA_{\cat{kac}}$ is precisely the largest quotient of $\cA$ on which $\tau_s$ acts trivially, i.e.~the quotient by the ideal generated by the elements
\begin{equation}\label{eq:kacelems}
  \tau_s(a)-a,\ s\in \bR,\ a\in \Pol(\bG).
\end{equation}
\end{proposition}
\begin{proof}
  On the one hand, since $(\tau_s)_{s\in\bR}$ is a one-parameter group of CQG automorphisms (i.e.~each $\tau_s$ preserves both the multiplication and the comultiplication), the quotient
  \begin{equation}\label{eq:atob}
    \cA\longrightarrow \cB
  \end{equation}
  by the ideal generated by \Cref{eq:kacelems} is indeed a CQG algebra. Since furthermore \Cref{eq:atob} intertwines scaling groups (\Cref{le:scalepres}) the scaling group of $\cB$ is trivial by construction and hence $\cB$ is Kac; this means that \Cref{eq:atob} factors as
  \begin{equation*}
    \xymatrix{\cA\ar[r]&\cA_{\cat{kac}}\ar[r]&\cB.}
  \end{equation*}
  On the other hand, the morphism $\cA\to \cA_{\cat{kac}}$ also intertwines scaling groups. Since its codomain has trivial scaling group, it must vanish on all elements of the form \Cref{eq:kacelems} and hence factor through $\cB$. In short, the kernels of $\cA\to \cA_{\cat{kac}}$ and \Cref{eq:atob} coincide.
\end{proof}

Next, the goal will be to prove

\begin{theorem}\label[theorem]{th:autocoam} 
  Let $\bG$ be a compact quantum group such that $\cA=\C(\bG)$ is type-I. Then $\bG$ is co-amenable.
\end{theorem}

Let $\cA=\C(\bG)$, as in the statement, i.e.~we assume that $\cA$ is type-I. By \cite[Main Theorem]{sakai-gcr} $\cA$ is \emph{GCR}, or \emph{postliminal} in the sense of \cite[Definition 4.3.1]{dix-cast}. Let
\begin{equation}\label{eq:ccrfilt}
  0=\cI_0\subset \cI_1\subset\cdots\subset \cI_{\alpha}=\cA
\end{equation}
the canonical transfinite composition sequence of ideals with postliminal subquotients $\cI_{\beta+1}/\cI_{\beta}$ provided by \cite[Proposition 4.3.3]{dix-cast}, where $\alpha$ is some ordinal number. We need

\begin{lemma}\label{le:haschar}
  A unital type-I \cst-algebra $\cA$ has at least one non-zero finite-dimensional irreducible representation.
\end{lemma}
\begin{proof}
  The ordinal $\alpha$ in \Cref{eq:ccrfilt} cannot be a limit ordinal: if it were, then by the very definition of a composition sequence (\cite[Definition 4.3.2]{dix-cast}) $\cI_{\alpha}=\cA$ would be the closure of the ascending chain of proper ideals $\cI_{\beta}$, $\beta<\alpha$, contradicting the fact that these are proper ideals in a \emph{unital} \cst-algebra and hence are all at distance $1$ from the unit $\I\in \cA$.

  It follows that $\alpha=\beta+1$ for some ordinal $\beta$, and hence the top liminal quotient
  $\cA/\cI_{\beta}$
  must be a (non-zero!) unital liminal \cst-algebra. It thus follows that all of its irreducible representations are finite-dimensional \cite[4.7.14]{dix-cast}.
\end{proof}

\begin{remark}
  Alternatively one could choose a proper maximal ideal $\cI$ in $\cA$ and note that then $\cA/\cI$ is a type-I simple unital \cst-algebra. Thus any representation of $\cA/\cI$ is faithful, and there exists an irreducible one, say $\phi:\cA/\cI\to\cB(H)$. The range of $\phi$ contains $\cK(H)$, and hence it must be equal to $\cK(H)$ (otherwise $\phi^{-1}(\cK(H))$ would be a proper ideal in $\cA/\cI$), but $\cA/\cI$ is unital and and $\phi$ is faithful, so $H$ must be finite dimensional.
\end{remark}

\begin{proof}[Proof of \Cref{th:autocoam}]
  Let $h$ be the Haar measure of $\bG$ and $\cJ$ the ideal
  \begin{equation*}
    \bigl\{x\in \cA\,\bigr|\bigl.\,h(x^*x)=0\bigr\}\subset \cA.
  \end{equation*}
  The quotient $\cA/\cJ$ will then be the reduced version $\Cr(\bG)$ and again of type I. Since it has a finite-dimensional representation by \Cref{le:haschar}, co-amenability follows from \Cref{th:coamcrit}.
\end{proof}

\section{Extensions of $\bK(\cH)$ by $\C(X)$}\label{se:cpctext}

Throughout the present section, $\cA$ denotes a \cst-algebra fitting into an exact sequence
\begin{equation}\label{eq:theext}
\xymatrix{0\ar[r]&\bK(\cH)\ar[r]&\cA\ar[r]^\pi&\cC\ar[r]&0}
\end{equation}
where
\begin{itemize}
  \item $\cC=\C(X)$ for a (non-empty and for us always Hausdorff) compact space $X$,
  \item the ideal $\bK(\cH)$ is as in \Cref{eq:ckh}, where
  \begin{equation}\label{eq:dimge2}
    \cH = \{H_{\lambda}\}_{\lambda\in\Lambda},\quad \dim H_{\lambda}\ge 2,\; \forall\:\lambda\in\Lambda
  \end{equation}
  is a finite, non-empty family of Hilbert spaces.
\end{itemize}
Note that such $\cA$ is automatically of type I and $\bK(\cH)$ is its CCR-ideal.
We denote
\begin{equation}\label{eq:dirsum}
  H:=\bigoplus_{\lambda \in \Lambda}H_{\lambda}.
\end{equation}

We list some examples of interest.

\begin{example}
  For any finite family \Cref{eq:dimge2}, the unitization $\bK(\cH)^+$ satisfies the hypotheses.
\end{example}

\begin{example}
  The \emph{Toeplitz \cst-algebra} $\cT(\partial D)$ (\cite[Definition 2.8.4]{hr-bk}) associated to a {\it strongly (or strictly) pseudoconvex domain} $\Omega\subset \bC^n$ (\cite[\S 3.2]{krntz} or \cite[Definition 1.2.18]{up-bk}) is of the form above, with $\cH$ a singleton.

  This applies in particular to the case when $D$ is the open unit disk in $\bC$. $\cT(\partial D)$ is then the universal \cst-algebra generated by an isometry, and \Cref{th:cpctext} below specializes to the main result of \cite{ks-nqg}.
\end{example}

\begin{example}
  The non-quotient \emph{Podle\'s spheres} introduced in \cite{pod-sph} and surveyed for instance in \cite[\S 2.5, point 5]{dab-gar}. According to \cite[Proposition 1.2]{sheu-quant} those algebras (denoted here collectively by $\cA$) are all isomorphic to the pullback of two copies of the symbol map $\cT\to\C(\bS^1)$. It follows that the \cst-algebra in question fits into an extension
  \begin{equation*}
    \xymatrix{0\ar[r]&\cK(\ell^2)\oplus \cK(\ell^2)\ar[r]&\cA\ar[r]&\C(\bS^1)\ar[r]&0}
  \end{equation*}
  i.e.~of the form \Cref{eq:theext} for a two-element family $\cH=\{H_{\lambda}\}_{\lambda\in\Lambda}$ of Hilbert spaces.
\end{example}

\begin{example}
  As recalled in \cite[Example, p.123]{lesch}, the algebra $\operatorname{CZ}(M)$ of \emph{Calder\'on-Zygmund operators} (i.e.~pseudo-differential operators of order zero; cf.~e.g.~\cite[\S VI.1]{stein}) on a smooth compact manifold $M$ fits into an exact sequence
  \begin{equation*}
    \xymatrix{0\ar[r]&\cK(L^2(M))\ar[r]&\operatorname{CZ}(M)\ar[r]&\C(S^*\!M)\ar[r]&0}
  \end{equation*}
  where $S^*\!M$ denotes the unit sphere bundle attached to the cotangent bundle of $M$.
\end{example}

With all of this in place, the main result of this section is

\begin{theorem}\label[theorem]{th:cpctext} 
 If $\bG$ is a compact quantum group such that a unital \cst-algebra $\cA=\C(\bG)$ fits into an exact sequence \Cref{eq:theext} as above, then $\cA$ is finite-dimensional.
\end{theorem}

\begin{remark}
  The discreteness hypothesis on the ideal $\bK(\cH)$ in \Cref{th:cpctext} is crucial: according to \cite[Appendix 2]{wor-su2}, for deformation parameters $\mu$ of absolute value $<1$ the function algebra $\C(\operatorname{SU}_{\mu}(2))$ fits into an exact sequence
  \begin{equation*}
    \xymatrix{0\ar[r]&\C(\bS^1)\otimes \cK(\ell^2)\ar[r]&\C(\operatorname{SU}_{\mu}(2))\ar[r]&\C(\bS^1)\ar[r]&0}.
  \end{equation*}
\end{remark}

\begin{remark}
  Let us also note that the fact that we are dealing with a \emph{unital} \cst-algebra $\cA$ is essential for \Cref{th:cpctext} as well. Indeed the \cst-algebras associated with the non-compact quantum ``$az+b$'' groups (\cite{azb,nazb}) are extensions of $\cK(H)$ by $\bC$ for an infinite dimensional separable Hilbert space $H$.
\end{remark}

\Cref{th:cpctext} will require some preparation. We will first address the issue of faithfulness of the Haar measure.

\begin{proposition}\label[proposition]{pr:coamen} 
  Let $\cA$ be an extension of $\bK(\cH)$ by $\C(X)$ as in \Cref{eq:theext} and suppose $\cA=\C(\bG)$ for some compact quantum group $\bG$. Then $\bG$ is co-amenable. In particular $\cA=\C(\bG)$ is reduced.
\end{proposition}
\begin{proof}
  This is a direct application of \Cref{th:autocoam}, since our \cst-algebra $\cA$ satisfies the hypotheses of that earlier result.
\end{proof}

We henceforth write $\cA=\Cr(\bG)$ to emphasize the faithfulness of the Haar measure, as allowed by \Cref{pr:coamen}.

Recall from \Cref{se:prel} that for each $\lambda$ we have the irreducible representation $\rho^\lambda:\cA\to\cB(H_{\lambda})$ obtained via the canonical extension of the embedding $\cK(H_{\lambda})\hookrightarrow\cB(H_{\lambda})$. Now in the present case $\{\rho^\lambda\}_{\lambda \in \Lambda}$ is precisely the subset of those irreducible representations of $\cA$ which are of dimension strictly greater than one. It follows that the subset $\{\rho^{\lambda}\}_{\lambda\in\Lambda}\subset \widehat{\cA}$ of the spectrum is invariant under every automorphism of $\cA$.\footnote{This also follows from a reasoning similar to the one in the proof of \Cref{le:autoinnergen}.} On the other hand, because that set is discrete in our case, each individual $\rho^{\lambda}$ is invariant under every one-parameter automorphism group of $\cA$. In other words, every one-parameter automorphism group of $\cA$ (e.g.~the modular group $(\sigma_t)_{t\in\bR}$ or the scaling group $(\tau_s)_{s\in\bR}$ coming from the CQG structure, for instance)
\begin{itemize}
  \item restricts to a one-parameter automorphism group of each ideal $\cK(H_{\lambda})$ of $\cA$, and also
  \item induces a one-parameter automorphism group of the image $\cA_{\lambda}$ of $\rho^{\lambda}$.
\end{itemize}

In this context, we have

\begin{lemma}\label{le:densop}
  On each $\cA_{\lambda}\subset\cB(H_{\lambda})$, the modular automorphism $\sigma_t$ of the Haar measure on $\cA$ acts as conjugation by $a_\lambda^{it}$ for some non-singular, positive, trace-class operator $a_\lambda$.
\end{lemma}
\begin{proof}
  The restriction of the Haar measure $h$ to $\cK(H_{\lambda})$ is of the form $\operatorname{Tr}\bigl(d^{\frac 12}\cdot d^{\frac 12}\bigr)$ for some positive, trace-class operator $d$ on $H_{\lambda}$. It follows that
  \begin{equation*}
    \bigl.\sigma_t\bigr|_{\cK(H_{\lambda})} = \text{conjugation by }d^{it}.
  \end{equation*}

  On the other hand, we know from \Cref{le:1paraminnergen} that
  \begin{equation*}
\bigl.\sigma_t\bigr|_{\cA_{\lambda}}  = \text{conjugation by }a_\lambda^{it}
  \end{equation*}
  for a possibly-unbounded non-singular positive self-adjoint operator $a_\lambda$ on $H_{\lambda}$. Since conjugation by $a_\lambda^{it}$ and $d^{it}$ agree on $\cK(H_{\lambda})$, the operators $a_\lambda$ and $d$ must be mutual scalar multiples. Finally, since $d$ is trace-class, so is $a_\lambda$.
\end{proof}

Recall that by \Cref{le:1paraminnergen}, on each $H_{\lambda}$
\begin{equation*}
\bigl.\tau_s\bigr|_{\cA_\lambda}  = \text{conjugation by }b_\lambda^{is}
\end{equation*}
for a possibly-unbounded positive self-adjoint non-singular operator $b_\lambda$ on $H_{\lambda}$. Moreover, because for each $s,t\in\bR$ the automorphisms $\tau_s$ and $\sigma_t$ commute, conjugation by $b_\lambda^{is}$ and $a_\lambda^{it}$ do too (with $a_\lambda$ as in \Cref{le:densop}).

\begin{proof}[Proof of \Cref{th:cpctext}]
  \sloppy
  If at least one of the spaces $H_{\lambda}$ is finite-dimensional then $\cA$ is finite-dimensional. Indeed, assume that $\dim(H_{\lambda})<+\infty$ for some $\lambda\in\Lambda$ and let $p\in \cA$ be the central projection corresponding to the unit of $\cK(H_{\lambda})$. Then $p\cA$ is a finite dimensional ideal in $\cA$ isomorphic to $\cK(H_{\lambda})=\cB(H_{\lambda})$. It follows that it is also a weakly closed ideal in $\li(\bG)\subseteq\cB(\mathnormal{L}^2(\bG))$, hence the claim is a consequence of \cite[Theorem 3.4]{dkss}.

  Due to the above observation, we assume all $H_{\lambda}$ are infinite-dimensional throughout the rest of the proof, and derive a contradiction. Observe that $\bG$ cannot be of Kac type as $\bK(\cH)$ has no faithful bounded traces.

  \begin{claim}
    The operator $\bigoplus\limits_{\scriptscriptstyle\lambda\in\Lambda}b_\lambda$ implementing the scaling group has finite spectrum.
  \end{claim}

  Assuming the claim for now, we can conclude by noting that since
  \begin{equation*}
    \rho\big(\tau_s(z)\bigr)=
    \bigoplus_{\lambda\in\Lambda}\; b_\lambda^{is}\, \rho_\lambda(z)\,b_\lambda^{-is}
  \end{equation*}
  for all $z\in \cA$ and $s\in\bR$, and operators $\bigoplus\limits_{\scriptscriptstyle\lambda\in\Lambda} b_\lambda$, $\bigoplus\limits_{\scriptscriptstyle\lambda\in\Lambda} b_\lambda^{-1}$ are bounded, the analytic generator $\tau_{-i}$ has bounded extension to all of $\cA$. It follows that $\bG$ is of Kac type (see \cite[discussion following Example 1.7.10]{nt-bk}), hence we arrive at a contradiction.

  It thus remains to prove the claim. We will do this with an argument similar to the one used in the proof of \cite[Theorem 13]{ks-nqg}. Let us for each $\alpha\in \Irr(\bG)$ choose a unitary representation $U^{\alpha}\in \alpha$ together with an orthonormal basis in the corresponding Hilbert space in which the positive operator $\uprho_{\alpha}$ is diagonal with entries
  \begin{equation*}
    \uprho_{\alpha,1},\dotsc,\uprho_{\alpha,n_{\alpha}}.
  \end{equation*}
  (cf.~\cite[Section 1.4]{nt-bk})
  Moreover, let $U^{\alpha}_{u,v}$ $(u,v\in\{1,\dotsc,n_\alpha\}$) be the corresponding matrix elements of $U^{\alpha}$. Recall that we have a quotient map
  \begin{equation*}
    \pi:\cA\longrightarrow \cC=\cA/\bK(\cH).
  \end{equation*}
  Clearly it factors through the canonical Kac quotient $\cA_{\cat{kac}}$ (\cite[Appendix]{sol-bohr}), hence thanks to the Lemma \ref{le:scalepres} we have $\pi\circ \tau_s=\pi$ for all $s\in \bR$. On the other hand, $\tau_s$ scales $U^{\alpha}_{u,v}$ by $\uprho_{\alpha,u}^{-is}\uprho_{\alpha,v}^{is}$, and hence non-trivially whenever $\uprho_{\alpha,u}\ne \uprho_{\alpha,v}$. Consequently
  \begin{equation*}
    \pi(U^{\alpha}_{u,v})=0 \quad\text{whenever}\quad \uprho_{\alpha,u}\ne \uprho_{\alpha,v}.
  \end{equation*}

  This means that upon applying $\pi:\cA\to \cC$, the matrix
  \begin{equation*}
    U^{\alpha}=
    \begin{bmatrix}
      U^\alpha_{1,1}&\dotsm&U^\alpha_{1,n_\alpha}\\
      \vdots&\ddots&\vdots\\
      U^\alpha_{n_\alpha,1}&\dotsm&U^\alpha_{n_\alpha,n_\alpha}
    \end{bmatrix}
  \end{equation*}
  becomes block-diagonal, with one block for each distinct value in the spectrum of $\uprho_{\alpha}$. Having relabeled that spectrum we can assume that
  \begin{equation*}
    \uprho_{\alpha,1},\dotsc,\uprho_{\alpha,d}
  \end{equation*}
  are all of the instances of a specific eigenvalue $\uprho>1$ in that spectrum. Now, define
  \begin{equation}\label{eq:truncmat}
  x=
  \begin{bmatrix}
    U^\alpha_{1,1}&\dotsm&U^\alpha_{1,d}\\
    \vdots&\ddots&\vdots\\
    U^\alpha_{d,1}&\dotsm&U^\alpha_{d,d}
  \end{bmatrix}
  \in \cB(\bC^d)\otimes \cB(H)=\cB(\bC^d\otimes H)
  \end{equation}
  to be the block of $U^{\alpha}$ corresponding to $\uprho$.

  The fact that the original matrix $U^{\alpha}$ was unitary and the above remark that off-diagonal $U^{\alpha}_{u,v}$ are annihilated by $\pi$ now imply that \Cref{eq:truncmat} is unitary mod $\cK(\bC^d\otimes H)$. In particular, the operator $x$ has finite-dimensional kernel by Atkinson's theorem (e.g. \cite[Theorem 3.3.2]{arv-spc}).

  Consider the operators
  \begin{equation*}
    A=    \I\otimes \bigl(\bigoplus_{\lambda\in \Lambda} a_\lambda\bigr)
    \quad\text{and}\quad
    B=    \I\otimes \bigl(\bigoplus_{\lambda\in \Lambda} b_\lambda\bigr)
  \end{equation*}
  acting on $\bC^d\otimes H$. Clearly they are positive, self-adjoint and non-singular, $A$ is bounded and $x$, $B^{is}$ commute for all $s\in\bR$. Furthermore, $A^{it},B^{is}$ commute for all $t,s\in\bR$. Indeed, it suffices to argue that $a_\lambda^{it},b_\lambda^{is}$ commute for each $\lambda\in\Lambda$. As conjugation by these unitary operators implements the modular and the scaling group on $\cA_\lambda$, we have
  \begin{equation*}
    a_\lambda^{it}b_\lambda^{is}a_\lambda^{-it}b_\lambda^{-is} = e^{i\frac{\hbar st}{2}},\quad \forall\:s,t\in \bR
  \end{equation*}
  for some fixed $\hbar\in \bR$ (see e.g.~\cite[p.5 and Definition 14.2]{hall}). If $\hbar\ne 0$ then according to the Stone-von Neumann theorem (\cite[Theorem 14.8]{hall}) there is a unitary operator from $H_\lambda$ onto $L^2(\bR)\otimes H_0$ (for some non-zero Hilbert space $H_0$) and identifying
  \begin{align*}
    a_\lambda&\longmapsto \exp\left(-i\hbar\tfrac{d}{dx}\right)\otimes \I_{H_0},\\
    b_\lambda&\longmapsto \bigl(\text{multiplication by $e^x$}\bigr)\otimes\I_{H_0}.
  \end{align*}
  Neither of these operators is bounded, hence we get a contradiction. It follows that $\hbar$ must vanish, so we can henceforth assume that $A$ and $B$ strongly commute. The following lemma gives us the claim and ends the proof.
\end{proof}

\begin{lemma}\label{le:isbd}
  On a Hilbert space $H$, let
  \begin{itemize}
    \item $a$ and $b$ be strongly commuting positive self-adjoint non-singular operators with $a$ bounded,
    \item $x$ be a bounded operator with finite-dimensional kernel, commuting with $b^{is}$ for all $s\in\bR$, and such that
    \begin{equation}\label{eq:qmute}
      a^{it} x a^{-it} = \uprho^{it}x,\quad \forall\:t\in \bR
    \end{equation}
    for some $\uprho>1$.
  \end{itemize}
  Then, $b$ has finite spectrum.
\end{lemma}

\begin{proof}
  Naturally, it suffices to assume $H$ is infinite-dimensional (otherwise there is nothing to prove).
  Let us denote by
  \begin{equation*}
    \text{Borel subsets of $\bR$}\ni\Omega\longmapsto E_{\Omega}\in\text{Projections on }H
  \end{equation*}
  the spectral resolution of $b$. If the latter has infinite spectrum, we could partition $\bR$ into infinitely many $\Omega_n$, $n\in \bZ_{\ge 0}$ with $E_n:=E_{\Omega_n}$ non-zero.

  Because $a$ and $b$ strongly commute, $a$ preserves the subspaces $H_n:=\operatorname{Im}(E_n)$ and thus admits a spectral resolution
  \begin{equation*}
    \Omega\longmapsto P_{n,\Omega}
  \end{equation*}
  thereon. By \Cref{eq:qmute} and the fact that $x$ and $b$ strongly commute, $x$ maps each range $\mathrm{Im}(P_{n,\Omega})$ to $P_{n,\uprho\Omega}$. The boundedness of $a$ means that we cannot scale by $\uprho>1$ indefinitely, so the kernel of $\bigl.x\bigr|_{H_n}$ is non-zero for all $n$. Since there are infinitely many summands $H_n$, we are contradicting the assumption on the finite-dimensionality of $\ker{x}$.
\end{proof}

\begin{remark}
  Due to the argument in the proof of \Cref{th:cpctext} showing that $A^{it}$ and $B^{is}$ commute, \Cref{le:isbd} in fact goes through under the formally weaker assumption that the conjugation actions by $a^{it}$ and $b^{is}$ commute on the algebra of compact operators.
\end{remark}



\providecommand{\bysame}{\leavevmode\hbox to3em{\hrulefill}\thinspace}
\providecommand{\MR}{\relax\ifhmode\unskip\space\fi MR }
\providecommand{\MRhref}[2]{%
  \href{http://www.ams.org/mathscinet-getitem?mr=#1}{#2}
}
\providecommand{\href}[2]{#2}

\end{document}